\documentclass[11pt,a4paper,reqno]{amsart}
\usepackage{amsaddr}
\usepackage{amsmath,amsfonts,amssymb,amsthm,mathrsfs}
\usepackage[utf8]{inputenc}
\usepackage[T1]{fontenc}
\usepackage{lmodern}
\usepackage{latexsym}
\usepackage{cancel}
\usepackage{microtype}
\usepackage{caption}
\usepackage{MnSymbol}
\usepackage{xcolor}
\usepackage{enumerate}
\usepackage[normalem]{ulem}
\usepackage{bbm}
\usepackage{geometry}
\usepackage{marginnote}
\usepackage{mathrsfs}

\usepackage[colorlinks=true,linkcolor=blue,citecolor=magenta]{hyperref}

\usepackage{graphics,tikz}

\newcommand{\GG}{{\mathcal  G}}

\newcommand{\MM}{{\mathcal  M}}

\newcommand{\BR}{{\mathbb R}}

\def\XXint#1#2#3{{\setbox0=\hbox{$#1{#2#3}{\int}$ }
\vcenter{\hbox{$#2#3$ }}\kern-.6\wd0}}

\newtheorem{theorem}{\bf Theorem}[section]
\newtheorem{proposition}[theorem]{\bf Proposition}
\newtheorem{lemma}[theorem]{\bf Lemma}

\theoremstyle{definition}
\newtheorem{definition}[theorem]{Definition}
\newtheorem{example}[theorem]{\bf Example}

\numberwithin{equation}{section}

\begin{document}

\title[Semi-linear elliptic equation with measure data]{Generalized solutions to  semilinear elliptic equations with measure data}

\maketitle
\begin{center}

  \normalsize
  TOMASZ KLIMSIAK\footnote{e-mail: {\tt tomas@mat.umk.pl}}\textsuperscript{1,2}
   \par \bigskip

  \textsuperscript{1} {\small Institute of Mathematics, Polish Academy Of Sciences,\\
 \'{S}niadeckich 8,   00-656 Warsaw, Poland} \par \medskip

  \textsuperscript{2} {\small Faculty of
Mathematics and Computer Science, Nicolaus Copernicus University,\\
Chopina 12/18, 87-100 Toru\'n, Poland }\par
\end{center}

\begin{abstract}
We address  an open problem posed by H. Brezis, M. Marcus and A.C. Ponce in:
{\em Nonlinear elliptic equations
with measures revisited. In: Mathematical Aspects of Nonlinear
Dispersive Equations (J. Bourgain, C. Kenig, S. Klainerman, eds.),
Annals of Mathematics Studies, {\bf 163} (2007)}. We prove that
for any bounded Borel measure $\mu$ on a smooth bounded domain $D\subset\mathbb R^d$
and asymptotically convex non-decreasing  non-negative  continuous function $g$ on $\mathbb R$
the sequence of solutions to the semi-linear   equation (P): $-\Delta u+g(u)=\rho_n\ast\mu$ ($\rho_n$ is a mollifier)
that is subject to homogeneous Dirichlet condition, converges to the function that solves (P)
with $\rho_n\ast\mu$  replaced by the {\em reduced measure} $\mu^*$ (metric projection onto the space of {\em good measures}).
We also provide a corresponding version of this  result  without non-negativity assumption on $g$.
\end{abstract}
\maketitle


\section{Introduction}
\label{sec1}
Let $D\subset \mathbb R^d$, $d\ge2$,
be a bounded domain with  smooth boundary,
$\mu$ be a bounded Borel measure on $D$ and
\begin{enumerate}
\item[(H)] $f:D\times\mathbb R\to\mathbb R$ be a Carath\'eodory function that is
non-increasing with respect to the second variable,
and $f(\cdot,y)\in L^1(D)$ for any $ y\in\mathbb R$.
\end{enumerate}
The present paper is  concerned with    the   Dirichlet problem  for nonlinear Poisson equation
\begin{equation}
\label{eq1.1}
-\Delta u=f(\cdot,u)+\mu \quad\text{in } D,\quad u=0\quad\text{on } \partial D.
\end{equation}

In 1975 (see the introduction in \cite{BB}) B\'enilan and  Brezis  discovered that in general  under merely condition (H)
there may not exist a solution to \eqref{eq1.1} even if $f$  admits a polynomial growth.
It appears  that if $\mu$ is above some level of concentration (determined by the Newtonian capacity)
then the study of \eqref{eq1.1} is highly non-trivial and non-existence phenomenon occurs.
On the other hand, in  many  interesting models equations of type \eqref{eq1.1} with polynomial or exponential growth absorption
term $f$  and highly concentrated measure $\mu$, as Dirac mass, appear (see, e.g., \cite{BB,LPY} and the references therein).
Brezis, Marcus and Ponce \cite[page 24]{BMP} posed the following natural problem. Consider the functions  $u_n$ solving the Dirichlet problems
\begin{equation}
\label{eq1.2}
-\Delta u_n=f(\cdot,u_n)+\rho_n\ast \mu \quad\text{in } D,\qquad u_n=0\quad\text{on } \partial D,
\end{equation}
where $(\rho_n)$ is a sequence  of smooth mollifiers. What can be said about the convergence of the sequence $(u_n)$ and, in case of convergence, about the form of the equation satisfied by the limit function?
Clearly, it cannot be \eqref{eq1.1} since in general there is no solution to \eqref{eq1.1}.
It is worth noting here that by  Stampacchia's inequality (see \eqref{eq2.4})
and the Rellich--Kondrachov theorem, $(u_n)$ is always convergent up  to a subsequence.

It turned out  that this is a quite difficult problem and  it   remained open to this day.
The subtlety of the problem is well exhibited by the fact that in general the limit of solutions $(u_n)$ to \eqref{eq1.2} with $\rho_n\ast \mu$ replaced by an approximation $(\mu_n)$ of $\mu$ in the narrow topology,
if  exists,
may vary depending on the choice of the sequence $(\mu_n)$ (see \cite[Remark 10.1]{MP}).

Interestingly, the problem is simplified when,  instead of regularization of $\mu$,   an approximation  $(f_n)$
of $f$ guaranteeing  the unique solvability of the problem
\[
-\Delta v_n=f_n(\cdot,v_n)+\mu \quad\text{in } D,\qquad v_n=0\quad\text{on } \partial D,
\]
is considered.
In \cite{BMP1,BMP} Brezis, Marcus and Ponce introduced the notion of  {\em reduced measure}.
They proved that under the additional assumption
They proved that under the additional assumption
\begin{enumerate}
\item[(B)] $f(x,y)=0,\ y\le 0$ $m$-a.e. $x\in D$
\end{enumerate}
there exists a maximal measure $\mu^{*,f}\le\mu$, called the reduced measure,
for which there exists a unique solution to
\begin{equation}
\label{eq.gw1}
-\Delta u=f(\cdot,u)+\mu^{*,f} \quad\text{in } D,\qquad u=0\quad\text{on } \partial D,
\end{equation}
and moreover, independently of the approximate  sequence $(f_n)$, $v_n\to u^{*,f}$
in $L^1(D)$. This legitimates referring to the unique solution of  \eqref{eq.gw1}  as a  generalized solution to \eqref{eq1.1}
(see comments preceding   \cite[Theorem 1]{BP}).
In \cite{BMP}  the authors  called {\em good measures}  (relative to $f)$ those bounded (signed) Borel measures for which  $\mu=\mu^{*,f}$. In
other words,  in their terminology, good measures are exactly those bounded Borel measures for which there exists a solution to  \eqref{eq1.1}. We denote the class of good measures  by $\GG(f)$.

The present paper is devoted to    the open problem posed by Brezis, Marcus and Ponce.
Let us mention that a partial answer to it has been already given in \cite[Theorem 4.11]{BMP}, where it is proved that if $f$ satisfies (B) and additionally $g:=-f$ is  convex and independent of the spatial variable, 
then
\[
u_n\to u^{*,f} \quad\text{in } L^1(D).
\]
Let $\mathcal M_b(D)$ denote the set of (signed) bounded Borel measures on $D$ equipped with the metric determined  by the total variation norm.
The main result
of the present paper states that  if   (H) and  the following condition:
\begin{enumerate}
\item[(UI)] for any function $w\in W^{1,q}_0(D),\, q\in [1,d/(d-1))$ that is  non-negative or non-positive in $D$ the following implication holds:
\[
\text{if} \quad f(\cdot,w)\in L^1(D)\quad \text{then}\quad (f(\cdot,\rho_n\ast w))_{n\ge 1}\quad\text{is uniformly integrable},
\] 
\end{enumerate}
are satisfied, then 
\[
u_n\to u^{\pi,f} \quad\text{in } L^1(D),
\]
where $u^{\pi,f}$ is the unique solution to the problem
\begin{equation}
\label{eq1.3}
-\Delta u=f(\cdot,u)+\Pi_f(\mu) \quad\text{in } D,\qquad u=0\quad\text{on } \partial D,
\end{equation}
and the mapping
\[
\Pi_f:\MM_b(D)\to\GG(f)
\]
is the unique continuous metric projection onto $\GG(f)$ such that
\[
\Pi_f(\mu+\nu)=\Pi_f(\mu)+\Pi_f(\nu)\quad\text{for any } \mu,\nu\in\MM_b(D) \text{ with } \mu\bot\nu.
\]
In particular, if $f$ satisfies (B)  (or, more generally, if there exists a subsolution to \eqref{eq1.1}),
then $\Pi_f(\mu)=\mu^{*,f}$. This means, in particular,  that the convergence proved in \cite[Theorem 4.11]{BMP} holds with the convexity assumption on $g$ replaced by (UI).
As a by-product of our results, we have that $(u_n)$ is convergent as a whole sequence.
Let us mention here that  if $g$ is independent of the spatial variable and  asymptotically convex  (recall that $g=-f$), i.e. 
\begin{enumerate}
\item[(AC)] there exists a convex function $\varphi:\mathbb R^+\to\mathbb R^+$ such that 
\[
\lim_{u\to\infty} \frac{g(u)}{\varphi(u)}=1,
\] 
\end{enumerate}
then (UI) holds (see Example \ref{ex1.11}). 

The usefulness of the aforementioned result, beyond its  theoretical value, is that it provides  practical tools for studying problems  of type \eqref{eq1.1}.
This is because we already have a fairly satisfactory  knowledge on the objects $\GG(f)$ and $\Pi_f$ involved in (\ref{eq1.3}).
Firstly, in \cite[Corollary 7.3]{K:arx} (see also \cite[Theorem 5.13]{K:CVPDE}) it has been proven that
\[
\overline{\mathcal A(f)}=\GG(f),
\]
where the closure is taken in the total variation norm, and
\begin{equation}
\label{eq1.4}
\mathcal A(f):=\{\mu\in \MM_b(D): f(\cdot,G_D\mu)\in L^1(D)\}.
\end{equation}
Here $G_D$ is Green's function for $D$ and $G_D\mu(x)=\int_D G_D(x,y)\,\mu(dy),\, x\in D$.
As a result,  for any function $g$ satisfying (H) the following implication holds:
\[
\mathcal A(f)=\mathcal A(g)\quad\Rightarrow \quad \Pi_f=\Pi_g.
\]
The antecedent of the above implication can be verified directly using \eqref{eq1.4}.
Secondly, from \cite[Theorem 6.23]{K:arx} (see also \cite[Theorem 4.15]{BMP}) it follows that for any $\mu\in\MM_b(D)$,
\[
\Pi_f(\mu)=(\mu^+)^{*,f}-(\mu^-)^{*,\tilde f},
\]
where $\tilde f(\cdot,y):= -f(\cdot,-y),\, x\in D, y\in\mathbb R$. Therefore, to calculate
$\Pi_f(\mu)$, it is enough to focus on the reduction operator $\mu\mapsto \mu^{*,f}$
that has been  studied in several papers (see, e.g., \cite{BLO,BMP,BP,DPP,K:CVPDE,K:arx,MP,PP}).

In the proof of the main result of the paper we utilize a series of  results on the  reduced measures and reduced limits
proved in \cite{BMP,MP} as well as the following fact proved recently in \cite[Theorem 7.2]{K:arx}: if $\mu\in\mathcal G(f)$,
then $w_n\to G^D\mu$ in $L^1(D)$ and $f(\cdot,w_n)/n\to 0$ in $L^1(D)$, where
\[
-\Delta w_n=\frac1n f(\cdot,w_n)+\mu,\qquad w_n=0\quad\text{on } \partial D.
\]
As a result, $\nu_n:=\frac1n f(\cdot,w_n)+\mu\in \mathcal A(f),\, n\ge 1$ and $\|\nu_n-\mu\|_\upsilon\le\frac1n\|f(\cdot,w_n)\|_{L^1(D)}\to 0$, where
$\|\mu\|_{\upsilon}$ stands for  the total variation norm of $\mu$.

\section{Notation and basic notions}

Throughout the paper, we fix a function  $f: D\times  \mathbb R\to \mathbb R$ that satisfies (H).
We denote by $\MM_b(D)$  the set of all bounded Borel measures on $D$,
and by $\MM^+_b(D)$  its subset consisting of positive measures ($\mu(A)\ge 0,\, A\in\mathcal B(D)$).
For $\mu\in\MM_b(D)$ we set $\|\mu\|_{\upsilon}:=|\mu|(D)$, where $|\mu|$ stands for the total variation measure of $\mu$
($|\mu|=\mu^++\mu^-$). The set $\MM_b(D)$ with the norm $\|\cdot\|_{\upsilon}$ is a Banach space.
We say that  $(\mu_n)\subset\MM_b(D)$ converges narrowly to $\mu\in\MM_b(D)$ if
\[
\int_D\eta\,d\mu_n\to \int_D\eta\,d\mu,\quad \eta\in C_b(D).
\]

\subsection{Definition of a solution, a priori estimates}

\begin{definition}
A function $u\in L^1(D)$ is a solution to \eqref{eq1.1} if
 $f(\cdot,u)\in L^1(D)$ and for any $\eta\in \mathcal C:=\{u\in C^2(\overline D): u=0 \text{ on } \partial D\}$,
\[
-\int_Du\Delta\eta=\int_D f(\cdot,u)\eta+\int_D\eta\,d\mu.
\]
\end{definition}

Below we recall some equivalent definitions.
A function  $u\in L^1(D)$ is a solution to \eqref{eq1.1} if and only if (see, e.g., \cite[Proposition 6.3]{Ponce})
$f(\cdot,u)\in L^1(D)$, $u\in W^{1,1}_0(D)$ and
for any $\eta\in C_c^\infty(D)$,
\[
-\int_Du\Delta\eta=\int_D f(\cdot,u)\eta+\int_D\eta\,d\mu.
\]
A function  $u\in L^1(D)$ is a solution to \eqref{eq1.1} if and only if (see, e.g., \cite[Theorem 1.2.2]{MV}), $f(\cdot,u)\in L^1(D)$ and
for a.e. $x\in D$,
\begin{equation}
\label{eq2.g1}
u(x)=\int_DG_D(x,y)f(y,u(y))\,dy+\int_DG_D(x,y)\,\mu(dy).
\end{equation}
Here $G_D$ denotes Green's function for $D$.

The following results are  well known.

\begin{proposition}
\label{prop2.2}
\begin{enumerate}[\rm(i)]
\item
For any solution $w$ to \eqref{eq1.1} and any $q\in [1,d/(d-1))$ we have
\begin{equation}
\label{eq2.4}
\|w\|_{W^{1,q}_0(D)}+\|f(\cdot,w)\|_{L^1(D)}\le C(\|f(\cdot,0)\|_{L^1(D)}+\|\mu\|_{TV}),
\end{equation}
where $C$ depends only on $q,d$ and $D$.
\item
Let $f_1, f_2$ satisfy \mbox{\rm(H)} and $\mu_1,\mu_2\in\MM_b(E)$. Let $u_1,u_2$ be
solutions to \eqref{eq1.1} with $(f,\mu)$ replaced by $(f_1,\mu_1), (f_2,\mu_2)$, respectively.
If $\mu_1\le\mu_2$ and $f_1(x,y)\le f_2(x,y)$ for any $y\in\mathbb R$ and a.e. $x\in D$, then $u_1\le u_2$ a.e.
\end{enumerate}
\end{proposition}
\begin{proof}
For (i) see, e.g., \cite[Appendix 4B]{BMP}, \cite[Proposition 5.1, Proposition 21.5]{Ponce},
\cite[Proposition 4.8]{K:CVPDE}). For (ii) see, e.g., \cite[Proposition 4.2]{K:CVPDE}, \cite[Corollary 4.B.2]{BMP}).
\end{proof}

\begin{definition}
A function $u\in L^1(D)$ is a subsolution (supersolution) to \eqref{eq1.1} if
 $f(\cdot,u)\in L^1(D)$ and
for any  $\eta\in \mathcal C^+:=\{\eta\in \mathcal C: \eta(x)\ge 0,\, x\in D\}$,
\[
-\int_Du\Delta\eta\le(\ge)\int_D f(\cdot,u)\eta+\int_D\eta\,d\mu.
\]
\end{definition}

\subsection{Reduced measures}

We denote by $\GG(f)$ the set of good measures (relative to $f$), i.e. the set of all measures $\mu\in\MM_b(D)$ for which there exists
a solution to \eqref{eq1.1}. It is well known that in general $\GG(f)\subsetneq \MM_b(D)$
(see \cite[Remark A.4]{BB}).

We denote by $\mbox{cap}_{H^1}$ the Newtonian capacity on $D$. It is well known (see, e.g., \cite[Lemma 4.A.1]{BMP}) that each measure $\mu\in\MM_b(D)$
admits the following unique decomposition
\[
\mu=\mu_d+\mu_c,
\]
where $\mu_d,\mu_c\in\MM_b(D)$ and $\mu_d\ll\mbox{cap}_{H^1}$, $\mu_c\bot\mbox{cap}_{H^1}$ (they are called the
{\em diffuse part} and the {\em concentrated part} of
$\mu$, respectively). $\MM^0_b(D)$ stands for the set of $\mu\in\MM_b(D)$ such that
$\mu=\mu_d$. By \cite[Corollary 4.B.3]{BMP} (see also \cite[Theorem 4.7]{KR:JFA}),
\[
\MM^0_b(D)\subset \GG(f),
\]
and  as a result, $L^1(D)\subset \GG(f)$ (the last inclusion, however, follows directly from \cite{BS, Konishi}).

Let $\GG_{\prec \mu}(f)$ denote the set of measures $\nu\in \GG(f)$ such that $\nu\le\mu$.
If $\GG_{\prec\mu}(f)\neq\emptyset$, then there exists $\mu^{*,f}\in\GG_{\prec\mu}(f)$ such that
\[
\max \GG_{\prec\mu}(f)=\mu^{*,f}
\]
(see, e.g.,  \cite[Theorem 4.15]{BMP}, \cite[Theorem 5.2]{K:CVPDE}, \cite[Theorem 5.2]{K:arx}). The measure $\mu^{*,f}$ is  called the {\em reduced measure}.
This notion was introduced in 2005  by Brezis, Marcus and Ponce  \cite{BMP}
and further generalized to non-local operators in \cite{K:CVPDE,K:arx}.

Since $f$ is a fixed,  throughout the paper  we  mostly  drop  the superscript $f$ on $\mu^{*,f}$  and write simply $\mu^*$.
Occasionally, however, we will use  full notation to emphasize the dependence of the reduction operator on  $f$.
In the sequel, we frequently use the following properties of the reduction operator:
\begin{equation}
\label{eq.propb}
(\mu^+)^*=(\mu^*)^+,\qquad (\mu_c)^*=(\mu^*)_c,\quad (\mu^*)_d=\mu_d,\qquad |\mu^*|\le |\mu|.
\end{equation}
For proofs we refer to \cite[Theorem 4.10, Corollary  4.10]{BMP} (see also \cite[Theorem 5.10, Proposition 5.4]{K:CVPDE}, \cite[Section 6.2]{K:arx}.

\section{Preparatory results}

Throughout the paper, we fix a smooth function $j:\mathbb R\to [0,\infty)$
such that $j(x)>0$ if $|x|<1$ and $j(x)=0$ if $|x|\ge 1$, and we let
\begin{equation}
\label{eq3.sm}
\rho_n(x):=cn^dj(n|x|),\quad x\in\mathbb R^d,
\end{equation}
where $c:=1/\int_0^1j(r)\alpha_d(r)\,dr$ and
$\alpha_{d}(r):= (2\pi^{d/2}r^{d-1})/\Gamma(d/2)$ (the surface area of the $d$-dimensional sphere of radius $r>0$).

For further study it will be convenient to introduce the following notion.
We denote  by $\mathcal G_{\#}(f,\mu)\subset\GG(f)$ the set  of all $\nu$ for which there exists
a subsequence $(n_k)$ such that $u_{n_k}\to u$ in $L^1(D)$, where
$u_{n_k}$  is the unique solution to
\begin{equation}
\label{eq2.1}
-\Delta u_{n_k}=f(\cdot,u_{n_k})+\rho_{n_k}\ast\mu \quad\text{in } D,\qquad u_{n_k}=0\quad\text{on } \partial D,
\end{equation}
and  $u$ is the unique solution to
\begin{equation}
\label{eq2.2}
-\Delta u=f(\cdot,u)+\nu \quad\text{in } D,\qquad u=0\quad\text{on } \partial D.
\end{equation}

In the remainder of this section, we also assume that $f$  satisfies condition (B)
formulated in the Introduction.

\begin{lemma}
\label{lm2.1}
For any $\nu\in \mathcal G_{\#}(f,\mu)$ we have $\nu\le \mu^*$.
\end{lemma}
\begin{proof}
Since $\nu\in \mathcal G_{\#}(f,\mu)$ there exists $(u_{n_k})$ and $u$ as described in \eqref{eq2.1}, \eqref{eq2.2}.
By the definition of  a solution, for any $\eta\in\mathcal C$,
\[
-\int_D u_{n_k}\Delta \eta-\int_D f(\cdot,u_{n_k})\eta=\int_D \rho_{n_k}\ast\eta\,d\mu.
\]
By \eqref{eq2.4} and Fatou's lemma,
\[
-\int_D u\Delta \eta-\int_D f(\cdot,u)\eta\le\int_D \eta\,d\mu,\quad \eta\in\mathcal C^+.
\]
On the other hand, by \eqref{eq2.2},
\[
-\int_D u\Delta \eta-\int_D f(\cdot,u)\eta=\int_D \eta\,d\nu,\quad \eta\in\mathcal C.
\]
Hence $\nu\le\mu$, and consequently  $\nu\le\mu^*$ since $\nu\in\mathcal G(f)$.
\end{proof}

\begin{proposition}
\label{prop3.2}
Suppose that  $\mathcal G_{\#}(f,\mu)=\{\mu^*\}$ for any  $\mu\in \MM^+_b(D)$.
Then $\mathcal G_{\#}(f,\mu)=\{\mu^*\}$ for any $\mu\in\MM_b(D)$.
\end{proposition}
\begin{proof}
Let $\mu\in \MM_b(D)$ and $\nu\in \mathcal G_{\#}(f,\mu)$. By the definition
of $\mathcal G_{\#}(f,\mu)$ there exists a subsequence $(n_k)$ such that $u_{n_k}\to u$ in $L^1(D)$,
where $u$ solves \eqref{eq2.2} and $u_{n_k}$ solves \eqref{eq2.1}.
Let $w_{n_k}$ be the unique solution to
\[
-\Delta u=f(\cdot,w_{n_k})+\rho_{n_k}\ast\mu^+ \quad\text{in } D,\qquad w_{n_k}=0\quad\text{on } \partial D.
\]
By the assumption that we made, $w_{n_k}\to w$ in $L^1(D)$, where $w$ solves
\[
-\Delta w=f(\cdot,w)+\mu^+ \quad\text{in } D,\qquad w=0\quad\text{on } \partial D.
\]
Consequently, by \cite[Theorem 7.1]{MP},
$
0\le (\mu^+)^*-\nu\le \mu^+-\mu.
$
Hence
\[
\mu^*=(\mu^+)^*-\mu^-\le\nu.
\]
This when combined with Lemma \ref{lm2.1} gives $\mu^*=\nu$.
\end{proof}

For any $\mu\in\MM_b(D)$ and $x\in D$ we let
\[
G_D\mu(x):= \int_D G_D(x,y)\,\mu(dy),
\]
whenever $\int_D G_D(x,y)\,|\mu|(dy)<\infty$ and zero otherwise.
Let us consider the following set of {\em admissible measures}:
\[
\mathcal A(f):= \{\mu\in\MM_b(D): f(\cdot,G_D\mu)\in L^1(D)\}.
\]

\begin{proposition}
\label{prop2.4}
Assume (UI). If $\mu\in\mathcal A(f)$, and $G_D\mu\ge 0$,  then  $\mathcal G_{\#}(f,\mu)=\{\mu\}$.
\end{proposition}
\begin{proof}
Let $(n_k)$ be a subsequence such that $u_{n_k}\to u$
for some $u\in L^1(D)$, where $u_{n_k}$ solves \eqref{eq2.1}.
Set $w:=G_D\mu\ge 0$.
By \eqref{eq2.g1} and the assumptions made on $f$, we have 
\[
u_{n_k}(x)=[G_Df(\cdot,u_{n_k})](x)+[G_D(\rho_{n_k}\ast\mu)](x)\le 
[G_D(\rho_{n_k}\ast\mu)](x)
\quad \text{in } D,\,\,m\text{-a.e.}
\]
(note that $f\le 0$ by (B)). Observe  that  $v_k:=G_D(\rho_{n_k}\ast\mu)$
solves 
\[
-\Delta v_k=\rho_{n_k}\ast\mu \quad\text{in } D,\qquad v_{k}=0\quad\text{on } \partial D,
\]
and $\rho_{n_k}\ast w $ solves 
\[
-\Delta (\rho_{n_k}\ast w) =\rho_{n_k}\ast\mu \quad\text{in } D,\qquad \rho_{n_k}\ast w \ge 0\quad\text{on } \partial D.
\]
Thus, $v_k\le \rho_{n_k}\ast w$, which combined with the previous inequality yields 
\[
u_{n_k}(x)\le
[\rho_{n_k}\ast w](x)\quad \text{in } D,\,\,m\text{-a.e.}
\]
Since $\mu$ was assumed to be in $\mathcal A(f)$, we have $w\in L^1(D)$.
By (UI),  $(f(\cdot,\rho_{n_k}\ast w))$ is uniformly integrable.
Consequently, by the Vitali   convergence theorem, $f(\cdot,u_{n_k})\to f(\cdot,u)$
in $L^1(D)$ (here we also used condition (B)). Therefore, letting $k\rightarrow\infty$ in   the equation
\[
-\int_Du_{n_k}\Delta \eta=\int_D f(\cdot,u_{n_k})\eta+\int_D(\rho_{n_k}*\eta)\,d\mu,\quad\eta\in\mathcal C,
\]
shows that $u$ solves \eqref{eq1.1}. Since the subsequence $(n_k)$ was chosen arbitrarly,
we conclude that $\mathcal G_{\#}(f,\mu)=\{\mu\}$.
\end{proof}

Consider the following condition (weaker than (AC))
\begin{enumerate}
\item[(A)] there exists a convex function $\varphi:\mathbb R^+\to\mathbb R^+$, and  $M,c_1,c_2>0$ such that 
\[
c_1\varphi(x)\le g(x)\le c_2\varphi(x),\quad x\ge M.
\]
\end{enumerate}

\begin{example}
\label{ex1.11}
In the present example we show that (A) implies (UI).
Assume that (A) holds and $f(w)\in L^1(D)$ for some non-negative $w\in L^1(D)$.
By the  de la Vall\'ee--Poussin lemma (see e.g.  \cite{Meyer})
 there exists  a convex increasing function
$\psi:\BR^+\rightarrow\BR^+$ such that  $\psi(0)=0$,  $\lim_{x\rightarrow\infty}\psi(x)/x=\infty$,   
$\psi(x+y)\le a\psi(x)+a\psi(y),\, x,y\in\mathbb R^+$ for some $a\ge 1$,  and 
\begin{equation}
\label{eq.ui1}
\int_D\psi(g(w))<\infty.
\end{equation}
By (A)
\[
c_1\varphi(u)-m\le g(u)\le c_2\varphi(u)+m,\quad u\ge 0,
\]
where $m:= \sup_{|u|\le M}\varphi(u)+\sup_{|u|\le M}g(u)$. Thus,
\[
\begin{split}
\psi\circ g(\rho_{n}\ast w)&\le c(a,c_2) \Big(\psi\circ\varphi(\rho_{n}\ast w)+\psi(m)\Big)
\\&\le c(a,c_2) \Big(\rho_{n}\ast [\psi\circ\varphi(w)]+\psi(m)\Big)\le c(a,c_1,c_2) \Big(\rho_{n}\ast [\psi\circ g(w)]+\psi(m)\Big).
\end{split}
\]
This combined  with \eqref{eq.ui1} yields
\[
\sup_{k\ge 1} \int_D\psi\circ g(\rho_{n}\ast w)<\infty.
\]
By the  de la Vall\'ee--Poussin lemma again $(f(\rho_{n}\ast w))_{n\ge 1}$ is uniformly integrable.
\end{example}

A  function $\phi:\mathbb R^+\to\mathbb R^+$ is said to satisfy $\Delta_2$-condition 
if there exists $C\ge 0$ such that $\phi(2x)\le C\phi(x),\, x\ge 0$.

\begin{example}
\label{ex1.12}
Observe that if $g$ satisfies $\Delta_2$-condition and there exists a strictly increasing convex 
function $\phi$, with $\phi(0)=0$, such that
$h(x):= g(x)/\phi(x)$ is  increasing, then (A) holds.
Indeed, assume first additionally that $\phi(x)=x,\, x\ge 0$. 
Observe that  the function
\[
\varphi_1(x):=\int_0^xg(y)y^{-1}\,dy,\quad x\ge 0
\]
is increasing and convex. 
Furthermore,
\[
g(x)\ge \varphi_1(x)\ge \int_{x/2}^xg(y)y^{-1}\,dy\ge g(x/2)\ge C^{-1}g(x).
\]
Now, applying the above inequality  to the function $g_\phi$ in place of $g$, where
\[
g_\phi(x):= g(\phi^{-1}(x)),\quad x\ge 0,
\]
we get 
\[
C^{-1}\varphi_1(\phi(x))\le g(x)\le \varphi_1(\phi(x)),\quad x\ge 0.
\]
Thus, letting $\varphi(x):= \varphi_1(\psi(x))$ we get (A).
\end{example}

\section{Main results}

\begin{theorem}
\label{th3.1}
Let $\mu\in \MM_b(D)$ and \mbox{\rm(B), (UI)} hold. Let $u_n$ be the unique solution to \eqref{eq1.2}
and $u^*$ be the unique solution to \eqref{eq.gw1}. Then
\[
\lim_{n\rightarrow\infty}\|u_n-u^*\|_{L^1(D)}=0.
\]
\end{theorem}
\begin{proof}
By Proposition \ref{prop3.2},  without loss of generality we may assume that $\mu\in\MM_b^+(D)$.

{\em Step 1}. Assume additionally that $\mu\in\GG(f)$.
 By \cite[Corollary 7.3]{K:arx}, for any $k\ge 1$ there exists a function
$g_k\in L^1(D)$ such that $\|g_k\|_{L^1(D)}\le 1/k$ and
\[
\mu-g_k\in \mathcal A(f),\quad G_D(\mu-g_k)\ge 0.
\]
Let $u$ be the solution to \eqref{eq1.1}, $u^k_n$ be the unique solution to
\[
-\Delta u^k_n=f(\cdot,u^k_n)+\rho_n\ast(\mu-g_k) \quad\text{in } D,\qquad u^k_n=0\quad\text{on } \partial D,
\]
and $u^k$ be the unique solution to
\[
-\Delta u^k=f(\cdot,u^k)+\mu-g_k \quad\text{in } D,\qquad u^k=0\quad\text{on } \partial D.
\]
By \eqref{eq2.4},
\[
\|u-u_n\|_{L^1(D)}\le  \|u-u^k\|_{L^1(D)}+\|u^k-u^k_n\|_{L^1(D)}+\|u^k_n-u_n\|_{L^1(D)}\le \frac{C}{k}+\|u^k-u^k_n\|_{L^1(D)}.
\]
For fixed $k\ge 1$, by Proposition \ref{prop2.4}, $\|u^k-u^k_n\|_{L^1(D)}\to 0$ as $n\to\infty$.
One easily concludes now that $\|u-u_n\|_{L^1(D)}\to 0$, which shows that $\mathcal G_{\#}(f,\mu)=\{\mu\}$.

{\em Step 2}. The general case. Let $\mu\in \MM_b^+(D)$. Since $\mu^*\le\mu$, we have
$\rho_n\ast\mu^*\le \rho_n\ast \mu,\, n\ge 1$. Let $w_n$ be the solution to
\[
-\Delta w_n=f(\cdot,w_n)+\rho_n\ast\mu^* \quad\text{in } D,\qquad w_n=0\quad\text{on } \partial D.
\]
Let $\nu\in \mathcal G_{\#}(f,\mu)$ and $(n_k)$ be a subsequence such that $u_{n_k}\to v$ in $L^1(D)$
and $v$ solves \eqref{eq2.2}. By {\em Step 1}, $w_{n_k}\to u^*$. Hence, by \cite[Theorem 7.1]{MP},
$0\le \nu-\mu^*$, so $\nu=\mu^*$ by Lemma \ref{lm2.1}.
\end{proof}

In \cite[Section 6.4]{K:arx} it is shown that there exists a continuous metric projection
onto $\GG(f)$
\[
\Pi_f:\MM_b(D)\to \GG(f),
\]
i.e.
\[
\inf_{\nu\in\GG(f)}\|\mu-\nu\|_{\upsilon}=\|\mu-\Pi_f(\mu)\|_{\upsilon}
\]
such that $\Pi_f(\mu+\nu)=\Pi_f(\mu)+\Pi_f(\nu)$ for any $\mu,\nu\in\MM_b(D)$, with $\mu\bot\nu$.
Moreover, there is at most  one continuous metric projection onto $\GG(f)$  having this property.
Furthermore, we have shown that $\Pi_f$ admits the following representation
\begin{equation}
\label{eq4.1}
\Pi_f(\mu)= (\mu^+)^{*,f}-(\mu^-)^{*,\tilde f}.
\end{equation}
where $\tilde f(x,y):= -f(x,-y),\, y\in\mathbb R,\, x\in D$.

\begin{lemma}
\label{lm3.1}
Let $\mu\in\MM^+_b(D)$. Then
\begin{enumerate}
\item[\rm(i)] For any functions $f_1,f_2$ satisfying \mbox{\rm(H)},
and such that $|f_1-f_2|\le h$ for some $h\in L^1(D)$, we have $\mu^{*,f_1}=\mu^{*,f_2}$.

\item[\rm(ii)] $\mu^{*,f}=\mu^{*,-f^-}$.
\end{enumerate}
\end{lemma}
\begin{proof}
Observe that $\mathcal A(f_1)=\mathcal A(f_2)$.
Hence, by \cite[Corollary 7.3]{K:arx}, $\mathcal G(f_1)=\mathcal G(f_2)$.
Since $\mu^{*,f_i}$ is the unique (see \cite[Corollary 4.6]{BMP}) metric projection onto $\mathcal G(f_i)$, $i=1,2$,
we get (i). As for (ii), we claim that for its proof  we may  assume without loss of generality that $f(\cdot,0)\equiv 0$.
Indeed, suppose  that (ii) holds  with $f$ replaced by   $\phi$ satisfying (H) and  such that $\phi(\cdot,0)\equiv 0$.
Set $f_0(x):=f(x,0),\, x\in D$.
Then by (i),
\[
\mu^{*,f}=\mu^{*,f-f_0}=\mu^{*,-(f-f_0)^-}=\mu^{*,-f^-}.
\]
This establishes the claim. Assume additionally that  $f(\cdot,0)\equiv 0$.
Let $u_n$ be the unique solution to
\[
-\Delta u_n=f(\cdot,u_n)\vee(-n)+\mu \quad\text{in } D,\qquad u_n=0\quad\text{on } \partial D,
\]
By \cite[Theorem 4.1]{BMP}, $u_n\to v$ in $L^1(D)$ and
\[
-\Delta v=f(\cdot,v)+\mu^{*,f} \quad\text{in } D,\qquad v=0\quad\text{on } \partial D.
\]
On the other hand, by  Proposition \ref{prop3.2} and the fact that $\mu$ is positive,
$u_n$ solves the problem
\[
-\Delta u_n=-f^-(\cdot,u_n)\vee(-n)+\mu \quad\text{in } D,\qquad u_n=0\quad\text{on } \partial D.
\]
Hence,  by \cite[Theorem 4.1]{BMP} again, $u_n\to w$, where
\[
-\Delta w=f(\cdot,w)+\mu^{*,-f^-} \quad\text{in } D,\qquad w=0\quad\text{on } \partial D.
\]
Clearly $v=w$, which implies that $\mu^{*,f}=\mu^{*,-f^-}$. This completes the proof of (ii).
\end{proof}

Combining Lemma \ref{lm3.1}(ii) with \eqref{eq4.1}, we obtain
\begin{equation}
\label{eq4.2}
\Pi_f(\mu)=(\mu^+)^{*,-f^-}-(\mu^-)^{*,\widetilde{f^+}}.
\end{equation}

Before  proceeding  to the proof of the main theorem, let us make
the following simple observations.
Let $\mu\in \GG(f)$. Then there exists a unique solution $u$ to \eqref{eq1.1}.
We therefore have
\[
-\Delta u=f^+(\cdot,u)+(\mu-f^-(\cdot,u)),\quad\text{in } D,\qquad u=0\quad\text{on } \partial D,
\]
and
\[
-\Delta u=-f^-(\cdot,u)+(\mu+f^+(\cdot,u)),\quad\text{in } D,\qquad u=0\quad\text{on } \partial D.
\]
Thus $\mu-f^-(\cdot,u)\in \GG(f^+)$, $\mu+f^+(\cdot,u)\in \GG(-f^-)$, so by \cite[Corollary 4.7]{BMP},
 $\mu\in \GG(f^+)$, $\mu\in \GG(-f^-)$. We may also write
 \[
-\Delta (-u)=\tilde f(\cdot,-u)-\mu,\quad\text{in } D,\qquad -u=0\quad\text{on } \partial D,
\]
which shows that $\mu\in\GG(f)$ if and only if $-\mu\in\GG(\tilde f)$.

\begin{theorem}
\label{th3.2}
Let $\mu\in \MM_b(D)$. Let $u_n$ be the unique solution to \eqref{eq1.2}
and $u^\pi$ be the unique solution to
\[
-\Delta u=f(\cdot,u)+\Pi_f(\mu) \quad\text{in } D,\qquad u=0\quad\text{on } \partial D.
\]
Then
\[
\lim_{n\rightarrow\infty}\|u_n-u^\pi\|_{L^1(D)}=0.
\]
\end{theorem}
\begin{proof}
Let $\nu\in\GG_{\#}(f,\mu)$.
By the very definition of the class $\GG_{\#}(f,\mu)$,  there exist  a subsequence $(n_k)$
and functions $u_{n_k}, u$   solving \eqref{eq2.1} and  \eqref{eq2.2}, respectively,
such that  $u_{n_k}\to u$ in $L^1(D)$.
Let $(v_{n_k})$ be the sequence of functions solving
\[
-\Delta u=-f^-(\cdot,v_{n_k})+\rho_{n_k}\ast\mu \quad\text{in } D,\qquad v_{n_k}=0\quad\text{on } \partial D,
\]
and $(w_{n_k})$ be the sequence of functions solving
\[
-\Delta w_{n_k}=f^+(\cdot,w_{n_k})+\rho_{n_k}\ast\mu \quad\text{in } D,\qquad w_{n_k}=0\quad\text{on } \partial D.
\]
Observe that $-w_{n_k}$ solves
\[
-\Delta(-w_{n_k})=\widetilde {f^+}(\cdot,-w_{n_k})+\rho_{n_k}\ast(-\mu) \quad\text{in } D,\qquad -w_{n_k}=0\quad\text{on } \partial D.
\]
By Proposition \ref{prop2.2},
$
v_{n_k}\le u_{n_k}\le w_{n_k}$, $k\ge 1$.
By Theorem \ref{th3.1}, $v_{n_k}\to v$ and $-w_{n_k}\to -w$, where
\[
-\Delta v=-f^-(\cdot,v)+\mu^{*,-f^-} \quad\text{in } D,\qquad v=0\quad\text{on } \partial D,
\]
and
\[
-\Delta (-w)=\widetilde {f^+}(\cdot,-w)+(-\mu)^{*,\widetilde {f^+}} \quad\text{in } D,\qquad -w=0\quad\text{on } \partial D.
\]
By the inverse maximum principle (see \cite[Proposition 7.2]{MP}),
\[
\mu^{*,-f^-}\le \nu\le -(-\mu)^{*,\widetilde{f^+}}.
\]
By \eqref{eq.propb},
\[
(\mu^{*,-f^-})_d=\mu_d,\quad   \big(-(-\mu)^{*,\widetilde{f^+}})\big)_d=\mu_d.
\]
Consequently, $\nu_d=\mu_d$.
Furthermore, by \cite[Lemm 4.1]{BMP} (see also \cite[Theorem 5.2]{K:CVPDE}),
\[
(\mu^{*,-f^-})_c=(\mu^{+}_c)^{*,-f^-}-\mu_c^-,\quad  (-\mu)^{*,\widetilde{f^+}}_c=(\mu^{-}_c)^{*,\widetilde{f^+}}-\mu^+_c,
\]
which implies that
\[
(\mu^{+}_c)^{*,-f^-}\le \nu_c^+\le \mu_c^+,\quad (\mu^{-}_c)^{*,\widetilde{f^+}}\le \nu_c^-\le \mu_c^-.
\]
Since $u$ solves \eqref{eq2.2} we have  $\nu\in \GG(f)$. Hence $\nu\in\GG(-f^-)$ and $-\nu\in\GG(\widetilde{f^+})$
(see the comments preceding the theorem).
Therefore, by \cite[Theorem 4.6']{BMP} (see also \cite[Theorem 5.11]{K:CVPDE}), $\nu^+_c\in\GG(-f^-)$ and $\nu^-_c\in\GG(\widetilde{f^+})$.
As a result,
\[
\nu_c^+=(\mu^{+}_c)^{*,-f^-},\quad \nu_c^-=(\mu^{-}_c)^{*,\widetilde{f^+}}.
\]
Hence, by \eqref{eq4.2},
\[
\nu=\nu_d+\nu_c=\mu_d+\nu_c=\mu_d+(\mu^{+}_c)^{*,-f^-}-(\mu^{-}_c)^{*,\widetilde{f^+}}=\Pi_f(\mu).
\]
This concludes the proof of the theorem.
\end{proof}

\subsection*{Acknowledgements}
{\small This work was supported by Polish National Science Centre
(Grant No. 2017/25/B/ST1/00878).}

\end{document}